\documentclass[10pt]{article}
\usepackage[cp1251]{inputenc}
\usepackage[english]{babel}
\usepackage{amsmath}
\usepackage{amsfonts}
\usepackage{amssymb}
\usepackage{graphicx}
\usepackage{hyperref}
\newtheorem{theorem}{\textsf{Theorem}}
\newtheorem{lemma}{\textsf{Lemma}}
\newtheorem{remark}{\textsf{Remark}}
\newenvironment{proof}[1][\textsf{Proof}]{\textbf{#1}}{$\square$}
\def\text{\hbox} 

\begin{document}

\title{
\date{ }
{
\large \textsf{\textbf{Spherical structures on torus knots and links}}
\thanks{The work is performed under auspices of the Swiss National Science Foundation no.~200020-113199/1,  ``Scientific Schools''-5682.2008.1 and RFBR no.~06-01-00153.}}}
\author{\small Alexander Kolpakov, Alexander Mednykh}
\maketitle

\begin{abstract}\noindent
The present paper considers two infinite families of cone-manifolds endowed with spherical metric. The singular strata is either the torus knot ${\rm t}(2n+1, 2)$ or the torus link ${\rm t}(2n, 2)$. Domains of existence for a spherical metric are found in terms of cone angles and volume formul{\ae} are presented.

\end{abstract}


\parindent=0pt

\section{Introduction}

A three-dimensional cone-manifold is a metric space obtained from a collection of disjoint simplices in the space of constant sectional curvature $k$ by isometric identification of their faces in such a combinatorial fashion that the resulting topological space is a manifold (also called the underlying space for a given cone-manifold).

Such the metric space inherits the metric of sectional curvature $k$ on the union of its 2- and 3-dimensional cells. In case $k=+1$ the corresponding cone-manifold is called spherical (or admits a spherical structure). By analogy, one defines euclidean ($k=0$) and hyperbolic ($k=-1$) cone-manifolds.

The metric structure around each 1-cell is determined by a cone angle that is the sum of dihedral angles of corresponding simplices sharing the 1-cell under identification. The singular locus of a cone-manifold is the closure of all its 1-cells with cone angle different from $2\pi$. For the further account we suppose that every component of the singular locus is an embedded circle with constant cone angle along it.

A particular case of cone-manifold is an orbifold with cone angles $2\pi/m$, where $m$ is an integer (cf.~\cite{T2}).

The present paper considers two infinite families of cone-manifolds with underlying space the three-dimensional sphere $\mathbb{S}^3$. The first family consists of cone-manifolds with singular locus the torus knot ${\rm t}(2n+1, 2)$ with $n\geq1$. In the rational census \cite{Ro} these knots are denoted by $(2n+1)/1$. The second family of cone-manifolds consists of those with singular locus a two-component torus link ${\rm t}(2n, 2)$ with $n\geq2$. These links are two-bridge and correspond to the links $2n/1$ in the rational census. The simplest examples of such the knots and links are the trefoil knot $3/1$ and the link $4/1$. In the Rolfsen table \cite{Ro} one finds them as the knot $3_1$ and the link $4_1^2$.

By the Theorem of W.~Thurston \cite{T1}, the manifold $\mathbb{S}^3\backslash3_1$ does not admit a hyperbolic structure.  However, it admits two other geometric structures \cite{N}: $\mathbb{H}^2 \times \mathbb{R}$ and $\widetilde{\rm{PSL}}(2,\mathbb{R})$. It follows from the paper \cite{SW} that the spherical dodecahedron space (i.e. Poincar\'e homology sphere) is a cyclic $5$-fold covering of $\mathbb{S}^3$ branched over $3_1$. Thus, the orbifold $3_1(\frac{2\pi}{5})$ with singular locus the trefoil knot and cone angle $\frac{2\pi}{5}$ is spherical. Due to the Dunbar's census \cite{D}, orbifold $3_1(\frac{2\pi}{n})$ is spherical if $n\leq5$, Nil-orbifold if $n=6$ and $\widetilde{\rm{PSL}}(2,\mathbb{R})$-orbifold if $n\geq7$. Spherical structure on the cone-manifold $3_1(\alpha)$ with underlying space the three-dimensional sphere $\mathbb{S}^3$ is studied in \cite{DMM}.

The consideration of two-bridge torus links is carried out starting with the simplest one possessing non-abelian fundamental group, namely $4^2_1$.

The previous investigation on spherical structures for cone-manifolds is carried out mainly in the papers \cite{HLM,MR2,Po}. The present paper develops a method to analyse existence of a spherical metric for two-bridge torus knot and link cone-manifolds. Also, the lengths of singular geodesics are calculated and the volume formul{\ae} are obtained (cf. Theorem~\ref{tor_knot} and Theorem~\ref{tor_link}).

\section{Projective model $\mathbb{S}^3_{\lambda}$}

The purpose of the present section is to construct the projective model $\mathbb{S}^3_{\lambda}$ that one can use to study geometric properties of two-bridge torus knots and links and to build up holonomy representation for the corresponding cone-manifolds. Other projective models for homogeneous geometries are described in \cite{Mo}.

Consider the set $\mathbb{C}^2=\{(z_1,\,z_2):z_1,
z_2\in\mathbb{C}\}$ as a four-dimensional vector space over $\mathbb{R}$. We denote it by $\mathbb{C}^2_\mathbb{R}$ and equip with Hermitian product
$$
{\langle(z_1,\,z_2),\,(w_1,\,w_2)\rangle}_{\rm{H}}=(z_1,\,z_2)\mathcal{H}\,\overline{(w_1,\,w_2)}^T,
$$
where
$$\mathcal{H}=\left(\begin{array}{ccc}1 & \lambda  \\\lambda & 1 \end{array}\right)$$ is a symmetric matrix with $-1\,<\,\lambda\,<\,+1$.

The natural inner product is associated to the Hermitian form above:
$$
\langle(z_1,\,z_2),\,(w_1,\,w_2)\rangle = {\rm Re}\,{\langle(z_1,\,z_2),\,(w_1,\,w_2)\rangle}_{\rm{H}}
$$
and the respective norm is
$$
\| (z_1, z_2) \| = |z_1|^2 + |z_2|^2 + \lambda (z_1 \overline{z}_2 + \overline{z}_1 z_2).
$$

Call two elements $(z_1,\,z_2)$ and $(w_1,\,w_2)$ in $\overset{\circ}{\mathbb{C} _{\mathbb{R}}^2}=\mathbb{C}^2_{\mathbb{R} } \,{\backslash}\,(0,\,0) $ equivalent if there is $\mu > 0$ such that $(z_1,\,z_2)=(\mu \,w_1,\,\mu\,w_2)$. We denote this equivalence relation as $(z_1,\,z_2)\sim(w_1,\,w_2).$

Identify the factor-space $\overset{\circ}{\mathbb{C} _{\mathbb{R}}^2}/\sim$ with the three-dimensional sphere $$\mathbb{S}^3_{\lambda}=\{ (z_1,\,z_2)\in \mathbb{C} _{\mathbb{R}}^2:\|(z_1,\,z_2)\| =1\},$$ endowed with the Riemannian metric
$$
{\rm d}s_{\lambda}^2=|{\rm dz_1}|^2+|{\rm dz_2}|^2+\lambda({\rm dz_1}{\rm d\overline{z}_2}+{\rm d\overline{z}_1}{\rm dz_2}).
$$

By means of equality
$$ {\rm d}s_{\lambda}^2=\frac{1+\lambda}2\,|{\rm d}z_1+{\rm d}z_2|^2+\frac{1-\lambda}2\,|{\rm d}z_1-{\rm d}z_2|^2,$$ the linear transformation
$$\xi_1=\sqrt{\frac{1+\lambda}2}\,(z_1+ z_2),\,\,\,\xi_2=\sqrt{\frac{1-\lambda}2}\,(z_1-z_2)$$ provides an isometry between $(\mathbb{S}^3_{\lambda},\,{\rm d}s_{\lambda}^2)$ and $(\mathbb{S}^3,\,{\rm d}s^2),$ where ${\rm d}s^2=|{\rm d\xi_1}|^2+|{\rm d\xi_2}|^2$ is the standard metric of sectional curvature $+1$ on the unit sphere $\mathbb{S}^3=\{(\xi_1,\,\xi_2)\in \mathbb{C}^2: |\xi_1|^2+ |\xi_2| ^2  = 1\}.$

Let $P,Q$ be two points in $\mathbb{S}^{3}_{\lambda}.$ The spherical distance between $P$ and $Q$ is a real number $d_{\lambda}(P,Q)$ that is uniquely determined by the conditions $0\leq d_{\lambda}(P,Q)\leq \pi $ and $\cos d_{\lambda}(P,Q) = \langle P,Q \rangle.$

\section{Torus knots $\mathbb{T}_n$}

Let $\mathbb{T}_n, n\geq1$ be the torus knot ${\rm t}(2n+1, 2)$ embedded in $\mathbb{S}^3$. The knot $\mathbb{T}_n$ is the two-bridge knot $(2n+1)/1$ in the rational census (Fig.~\ref{Fig1}). Let $\mathbb{T}_n(\alpha)$ denote a cone-manifold with singular locus $\mathbb{T}_n$ and the cone angle $\alpha$ along it.

\begin{figure}[ht]
\begin{center}
\includegraphics* [totalheight=4cm]{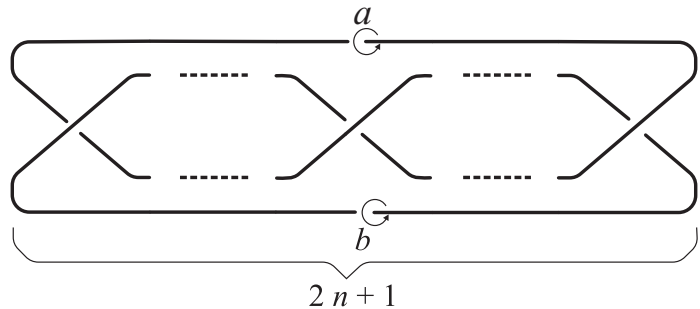}
\end{center}
\caption{Knot $(2n+1)/1$} \label{Fig1}
\end{figure}

The aim of the present section is to investigate cone-manifolds $\mathbb{T}_n(\alpha)$, $n\geq1$ to find out the domain of sphericity in terms of the cone angle and to derive the volume formul{\ae}.

Two lemmas precede the further exposition:

\begin{lemma} \label{isometry_lemma1}
For every $0<\alpha<2\pi$ and $-1<\lambda<+1$
the linear transformations $$A=\left(\begin{array}{ccc}1 & 0  \\-2\,i\,e^{i \frac{\alpha}{2}}\lambda\,\sin\frac{\alpha}{2} &e^{i\alpha} \end{array}\right)$$ and
$$B=\left(\begin{array}{ccc}e^{i\alpha}& -2\,i\,e^{i \frac{\alpha}{2}}\lambda\,\sin\frac{\alpha}{2}  \\ 0 & 1 \end{array}\right)$$
are isometries of $\mathbb{S}^3_{\lambda}$.
\end{lemma}

\begin{proof}. \label{proof_isometry_lemma1}
For the further account let us assume that the multiplication of vectors by matrices is to the right. A linear transformation $L$ of the space $\mathbb{C}^2_\mathbb{R}$ preserves the corresponding Hermitian form if and only if for every pair of vectors $P,Q\in\mathbb{C}^2_{\mathbb{R}}$ it holds that
$${\langle P,\,Q \rangle}_{\rm{H}}=P\mathcal{H}\,\overline{Q}^T=PL\mathcal{H}\overline{L}^T\overline{Q}^T={\langle P L,\,Q L\rangle}_{\rm{H}}.$$
The condition above is equivalent to
$$\mathcal{H}=L\mathcal{H}\overline{L}^T.$$

In particular,
$$\cos d_{\lambda}(P,Q)=\langle P,\,Q \rangle=\langle P L,\,Q L \rangle=\cos d_{\lambda}(P L,Q L),$$
that means $L$ preserves the spherical distance between $P$ and $Q$.

Let $L=A$ and $L=B$ in series, one verifies that $A$ and $B$ preserve the Hermitian norm on $\mathbb{C}^2_\mathbb{R}$ and, consequently, the spherical distance on $\mathbb{S}^3_{\lambda}$.
\end{proof}

\begin{lemma} \label{polynomial_lemma1}
Let $A$ and $B$ be the same matrices as in the affirmation of Lemma~\ref{isometry_lemma1}.
Then for all integer $n\geq1$ one has
$$
(AB)^nA-B(AB)^n\,=\,2\,U_{2n}(\Lambda)\,e^{i\,\frac{(2n+1)(\pi+\alpha)}{2}}\sin\frac{\alpha}{2}\,M,
$$where $M$ is a non-zero $2\times2$-matrix and $U_{2n}(\Lambda)$ is the second kind Chebyshev polynomial of power $2n$ in variable $\Lambda\,=\,\lambda\,\sin\frac{\alpha}{2}$.
\end{lemma}

\begin{proof}.
As far as $-1<\lambda<+1$, one obtains
$$-1\,<\,\Lambda\,=\,\lambda\,\sin\frac{\alpha}{2}\,<\,+1.$$
Substitute
$$\Lambda\,=\,\cos\theta,$$
with the unique $0\,<\,\theta\,<\,\pi$.

Then matrices $A$ and $B$ are rewritten in the form
$$ A=\left(\begin{array}{ccc}1 & 0  \\-2\,i\,e^{i \frac{\alpha}{2}}\cos\theta &e^{i\alpha} \end{array}\right), $$
$$ B=\left(\begin{array}{ccc}e^{i\alpha}& -2\,i\,e^{i \frac{\alpha}{2}}\cos\theta  \\ 0 & 1 \end{array}\right). $$

On purpose to diagonalize the matrix $AB$, use
$$ V=\left(\begin{array}{ccc}i\,e^{-i\,\frac{\alpha}{2}}\,e^{-i\theta} &  i\,e^{-i\,\frac{\alpha}{2}}\,e^{i\theta} \\1&1
\end{array}\right), $$
and obtain
$$ D\,=\,V^{-1}(AB)V\,=\,\left(\begin{array}{ccc}-e^{i\alpha}\,e^{2 i\theta} &0 \\0 &-e^{i\alpha}\,e^{-2 i\theta}\end{array}\right).$$

Note, that $V$ might be not an isometry, but it is utile for computation.

Thus
$$(AB)^nA-B(AB)^n\,=\,(V\,D^n\,V^{-1})A-B(V\,D^n\,V^{-1})\,=$$
$$=\,2\,\frac{\sin(2n+1)\theta}{\sin\theta}\,\,e^{i\,\frac{(2n+1)(\pi+\alpha)}{2}}\,\sin\frac{\alpha}{2}\,\,\left(\begin{array}{ccc}-1 &\lambda \\ -\lambda &1\end{array}\right)\,=$$
$$=\,2\,U_{2n}(\cos\theta)\,e^{i\,\frac{(2n+1)(\pi+\alpha)}{2}}\,\sin\frac{\alpha}{2}\,\,M\,=\,2\,U_{2n}(\Lambda)\,e^{i\,\frac{(2n+1)(\pi+\alpha)}{2}}\,\sin\frac{\alpha}{2}\,\,M,$$
with the matrix
$$
M = \left(\begin{array}{ccc}-1 &\lambda \\ -\lambda &1\end{array}\right)
$$
as the present Lemma claims.
\end{proof}

The main theorem of the section follows:
\begin{theorem} \label{tor_knot}
The cone-manifold $\mathbb{T}_n(\alpha)$, $n\geq 1$ is spherical if
$$\frac{2n-1}{2n+1}\,\pi < \alpha< 2\pi - \frac{2n-1}{2n+1}\,\pi.$$
The length of its singular geodesic (i.e. the length of the knot $\mathbb{T}_n$) equals
$$l_{\alpha}\,=\,(2n+1)\,\alpha\,-\,(2n-1)\,\pi.$$
The volume of $\mathbb{T}_n(\alpha)$ is
$$ {\rm \mathbb{V}ol}\,\mathbb{T}_n(\alpha)\,=\,\frac{1}{2n+1}\,\left(\frac{2n+1}{2}\,\alpha\,-\,\frac{2n-1}{2}\,\pi\right)^2.$$
\end{theorem}

\begin{proof}.
The fundamental group of the knot $\mathbb{T}_n$ is presented as
$$\pi_1(\mathbb{S}^3\backslash\mathbb{T}_n)=\langle a,b|(ab)^na\,=\,b(ab)^n \rangle,$$
with generators $a$ and $b$ as at Fig.~\ref{Fig1}.

Since the cone-manifold $\mathbb{T}_n(\alpha)$ admits a spherical structure, then there exists a holonomy mapping \cite{T2}, that is a homomorphism
$$h: \pi_1(\mathbb{S}^3\backslash\mathbb{T}_n) \longmapsto {\rm Isom}\,\,\mathbb{S}^3_{\lambda}.$$
We will choose $h$ in respect with geometric construction of the cone-manifold.

All the further computations to find the length of the knot $\mathbb{T}_n$ and the volume of the cone-manifold $\mathbb{T}_n(\alpha)$ are performed making use of the corresponding fundamental polyhedron $\mathcal{P}_n$ (Fig.~\ref{Fig2}). The construction algorithm for the polyhedron is given in \cite{MR1}.

The combinatorial polyhedron $\mathcal{P}_n$ has vertices $P_i$, $i\in\{1,\ldots,4n+2\}$ and edges $P_iP_{i+1}$, $i\in\{1,\ldots,4n+2\}$, with $P_{4n+3}=P_1$, also $P_1P_{2n+2}$ and $P_2P_{2n+3}$. Let $N$, $S$ denote the middle points (the North and the South poles of $\mathcal{P}_n$) on the edges $P_1P_{2n+2}$ and $P_2P_{2n+3}$, respectively. Then, consider also edges $NP_i$, $SP_i$, $i\in\{1,\ldots,4n+2\}$.

Without loss in generality, choose the holonomy representation such that
$$h(a) = A,\,\, h(b)=B,$$
where $A$ and $B$ are matrices from Lemma \ref{isometry_lemma1}.

The generators of the fundamental group for $\mathbb{T}_n$ under the holonomy mapping $h$ correspond to isometries acting on $\mathcal{P}_n$. These isometries identify its faces by means of rotation about the edge $P_1P_{2n+2}$ for the top ``cupola'' of $\mathcal{P}_n$ and rotation about $P_2P_{2n+3}$ for the bottom one (see, Fig.~\ref{Fig2}). Then the edges $P_1P_{2n+2}$ and $P_2P_{2n+3}$ knot itself to produce $\mathbb{T}_n$ (cf.~\cite{MR1,Min}).

In order to construct the polyhedron $\mathcal{P}_n$ assume that its edge $P_1P_2$ is given by
$$ P_1=(1,0),\,P_2=(0,1).$$
Then one has
$$ \cos d_{\lambda}(P_1,P_2) = \langle P_1,P_2 \rangle = \lambda, $$
i.e. the spherical distance between the points $P_1$ and $P_2$ can vary from $0$ to $\pi$. Thus, prescribing certain coordinates to the end-points of the edge $P_1P_2$ we do not loss in generality of the consideration.

Note, that the axis of the isometry $A$ from Lemma \ref{isometry_lemma1} contains $P_1$ and the axis of $B$ contains $P_2$. The aim of the construction for the polyhedron $\mathcal{P}_n$ is to bring its edges $P_1P_{2n+2}$ and $P_2P_{2n+3}$ to be axes of the respective isometries $A$ and $B$. The other vertices $P_i$ has to be images of $P_1$ and $P_2$ under action of $A$ and $B$.

\begin{figure}[ht]
\begin{center}
\includegraphics* [totalheight=8cm]{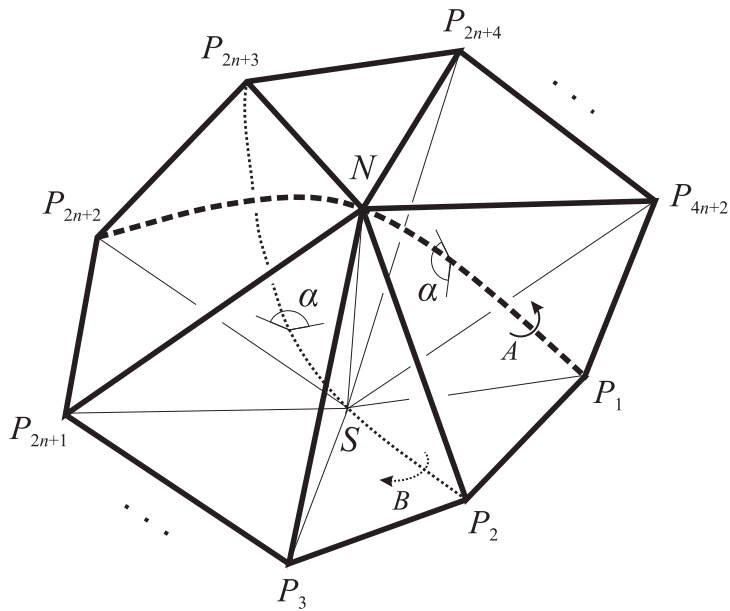}
\end{center}
\caption{Fundamental polyhedron $\mathcal{P}_n$ for $\mathbb{T}_n(\alpha)$} \label{Fig2}
\end{figure}

The polyhedron $\mathcal{P}_n$ is said to be proper if
\begin{list}{}{}
\item (a) inner dihedral angles along $P_1P_{2n+2}$ and $P_2P_{2n+3}$ are equal to $\alpha$;
\item (b) the following curvilinear faces are identified by $A$ and $B$:
$$A : NP_1P_2\ldots P_{2n+2} \rightarrow NP_1P_{4n+2}\ldots P_{2n+3}P_{2n+2}, $$
$$B : SP_2P_1P_{4n+2}\ldots P_{2n+3} \rightarrow SP_2P_3\ldots P_{2n+3}; $$
\item (c) sum of the inner dihedral angles $\psi_i$ along $P_iP_{i+1}$, $i\in~\{1,\ldots,4n+1\}$ equals $2\pi$;
\item (d) sum of the dihedral angles $\phi_i$ for corresponding tetrahedra $NSP_iP_{i+1}$, $i\in\{1,\ldots,4n+1\}$ at their common edge $NS$ is $2\pi$;
\item (e) all the tetrahedra $NSP_iP_{i+1}$ with $i\in\{1,\ldots,4n+2\}$, $P_{4n+3}=P_1$ are non-degenerated and coherently oriented.
\end{list}

By the orientation of a tetrahedron $NSP_iP_{i+1}$ one means the sign of the Gram determinant $\det(S, N, P_i, P_{i+1})$ for corresponding quadruple $S$, $N$, $P_i$, $P_{i+1}\in \mathbb{C}^2_\mathbb{R}$, where $i\in\{1,\ldots,4n+2\}$, $P_{4n+3}=P_1$. A tetrahedron is non-degenerated if $\det(S, N, P_i, P_{i+1})\neq0$. Thus, claim (e) is satisfied if all the Gram determinants are non-zero and of the same sign.

If $\alpha = \frac{2\pi}{m}$, $m\in\mathbb{N}$, then due to the Poincar\'e Theorem \cite[Theorem~13.5.3]{Ra} claims (a)~-- (e) imply that the group generated by the isometries $A$ and $B$ is discreet and its presentation is
$$
\Gamma = \langle A,B | (AB)^nA = B(AB)^n, A^m = B^m = {\rm id} \rangle.
$$

The metric space $\mathbb{S}^3_{\lambda}/\Gamma\,\cong\,\mathbb{T}_n(\frac{2\pi}{m})$ is a spherical orbifold, and $\mathcal{P}_n$ is its fundamental polyhedron. If $m\notin\mathbb{N}$ then the group generated by $A$ and $B$ might be non-discreet. However, the identification for the faces of $\mathcal{P}_n$ is of the same fashion as if it were $m\in\mathbb{N}$ and as the result one obtains the cone-manifold~$\mathbb{T}_n(\alpha)$.

By means of Lemma \ref{isometry_lemma1} and construction of $\mathcal{P}_n$ claims (a) and (b) are satisfied.

For the holonomy mapping $h$ to exist the following relation should be satisfied:
$$ h((ab)^na)-h(b(ab)^n)\,=\,(AB)^nA-B(AB)^n\,=\,0. $$

By Lemma \ref{polynomial_lemma1}, the condition above is satisfied if and only if
$$ U_{2n}(\Lambda)=0, $$
where $\Lambda=\lambda\sin\frac{\alpha}{2}\,$.

Thus, the parameter $\lambda$ of the metric ${\rm d}s_{\lambda}^2$ is determined completely by a root  of the polynomial $U_{2n}(\Lambda)$. From the above formula, $\lambda$ is related to the cone angle $\alpha$ by means of the equality 
$$\lambda\,=\,\frac{\Lambda}{\sin\frac{\alpha}{2}}\,.$$

The roots of $U_{2n}(\Lambda)$ are given by the following formula:

$$\Lambda_{k}\,=\,\cos\frac{k\pi}{2n+1}\,,$$
with $k\in\{1,\ldots,2n\}$.

The parameter $\lambda$ for the metric ${\rm d}s_{\lambda}^2$ has to be chosen in order the polyhedron $\mathcal{P}_n$ be proper and the metric itself be spherical.

Note, that the edges $P_iP_{i+1}$, $i\in\{1,\ldots,4n+2\}$, $P_{4n+3}=P_1$ are equivalent under action of the group $\Gamma=\langle A,B \rangle$. Thus, the relation $(AB)^nA = B(AB)^n$ implies the equality
$$ \sum^{2(2n+1)}_{i=1}\psi_{i}\,=\,2k\pi,$$
where $k$ is an integer.

Show that one can choose $\lambda$ for the equality $k=1$ to hold for all $\alpha$ in the affirmation of the Theorem. Due to the paper \cite{HR}, every two-bridge knot cone-manifold with cone angle $\pi$ is a spherical orbifold. In this case all the vertices $P_i$ of the fundamental polyhedron belong to the same circle and all the dihedral angles $\psi_{i}$ and $\phi_i$ are equal to each other \cite{MR1}:
$$ \phi_i = \psi_i = \frac{\pi}{2n+1}\,. $$

As far as $\cos d_{\lambda}(N,S)\,=\,\cos d_{\lambda}(P_i,P_{i+1})\,=\,\lambda$, then in case $\alpha=\pi$ one obtains
$$\lambda = \frac{\Lambda_k}{\sin\frac{\pi}{2}} = \cos\theta$$
for certain $k\in\{1,\ldots,2n\}$ and then
$$\sum^{2(2n+1)}_{i=1}\psi_{i}\,=\,2(2n+1)\theta.$$

Using the formula for the roots of $U_{2n}(\Lambda)$ obtain that
$$\sum^{2(2n+1)}_{i=1}\psi_{i}\,=\,2k\pi$$
if $\alpha=\pi$.
Thus, claim (c) for the polyhedron $\mathcal{P}_n$ with $\alpha=\pi$ is satisfied if $k=1$. As far as the parameter $\alpha$ varies continuously and sum of the angles $\psi_i$ represents a multiple of $2\pi$, one has that
$$\sum^{2(2n+1)}_{i=1}\psi_{i}\,=\,2\pi$$
for all $\alpha$.

By analogy, show that with
$$\lambda = \frac{\Lambda_1}{\sin\frac{\alpha}{2}}$$
the equality
$$\sum^{2(2n+1)}_{i=1}\phi_{i}\,=\,2\pi$$
holds, that means claim (d) is also satisfied.

Verify that under conditions of the Theorem the metric ${\rm d}s_{\lambda}^2$ is spherical. This claim is equivalent to the inequality
$$-1<\lambda<+1.$$
Note, that for
$$\frac{2n-1}{2n+1}\,\pi<\alpha<2\pi-\frac{2n-1}{2n+1}\,\pi$$
it follows
$$\sin\frac{\alpha}{2}\,>\,\sin\frac{(2n-1)\pi}{2(2n+1)}\,.$$
As far as $\sin\frac{\alpha}{2}>0$ and $\Lambda_1\,=\,\sin\frac{(2n-1)\pi}{2(2n+1)}>0$, one has
$$0<\lambda<1.$$

By analogy with Lemma \ref{isometry_lemma1} verify that
$$ C = \left(\begin{array}{ccc}0&1 \\1&0\end{array}\right) $$
is an isometry of ${\rm d}s_{\lambda}^2$.

Fixed point sets of $A$ and $B$ in $\mathbb{S}^3_{\lambda}$ are circles
$${\rm Fix}\,A = \{(z_1,0): z_1\in\mathbb{C}, |z_1|=1\}$$
and
$${\rm Fix}\,B = \{(0,z_2): z_2\in\mathbb{C}, |z_2|=1\},$$
correspondingly. The geometric meaning of $C$ is that it maps the first fixed circle to the other. Thus, the relation $B = C A C^{-1}$ holds.

The following equalities
$$P_{2k+1}\,=\,P_1(AB)^k,\,k\in\{0,\ldots,n\},$$
$$P_{2k}\,=\,P_2(AB)^{k-1},\,k\in\{1,\ldots,n+1\};$$
and
$$P_{2k+1}\,=\,P_1(BA)^{2n-k+1},\,k\in\{n+1,\ldots,2n\},$$
$$P_{2k}\,=\,P_2(BA)^{2n-k+2},\,k\in\{n+2,\ldots,2n+1\},$$
follow from the identification scheme of the edges of $\mathcal{P}_n$.

Define the auxiliary function
$${\varepsilon}(m) = \frac{m}{2}\,\alpha - \frac{4n-m}{2}\,\pi.$$

By analogy with the proof of Lemma~\ref{polynomial_lemma1} it follows that
$$(AB)^k\,=\,C(BA)^kC^{-1}\,=$$
$$=\,\left(\begin{array}{ccc}-\frac{\sin(2k-1)\theta}{\sin\theta}\,e^{i\,\varepsilon(2k)} &-\frac{\sin 2k\theta}{\sin\theta}\,e^{i\,\varepsilon(2k-1)} \\\frac{\sin2k\theta}{\sin\theta}\,e^{i\,\varepsilon(2k+1)} &\frac{\sin(2k+1)\theta}{\sin\theta}\,e^{i\,\varepsilon(2k)} \end{array}\right),$$
where $\theta=\frac{\pi}{2n+1}$.

Suppose $N$ and $S$ to be middle-points of the edges $P_1P_{2n+2}$ and $P_2P_{2n+3}$, respectively. Then
$$N = (e^{i\,\frac{\varepsilon(2n+1)}{2}},\,0),\,\,S =  (0,\,e^{i\,\frac{\varepsilon(2n+1)}{2}}).$$

For the lengths $l_{\alpha}$ of the singular geodesic one has
$$\cos\frac{l_{\alpha}}{4} = \langle P_1,N \rangle = \langle P_1 C, N C \rangle = \langle P_2,S \rangle.$$
Thus
$$
\cos\frac{l_{\alpha}}{4}\,=\,\cos\frac{(2n+1)\alpha-(2n-1)\pi}{4}\,.
$$

By construction of the polyhedron $\mathcal{P}_n$, the inequality $0<l_{\alpha}<4\pi$ holds. Then it follows 
$$l_{\alpha}=(2n+1)\alpha-(2n-1)\pi.$$

Given the coordinates of the vertices $P_i$ and the poles $N$ and $S$ of the polyhedron $\mathcal{P}_n$, verify claim (e).

For every four points $A,B,C,D \in \mathbb{C}^2_\mathbb{R}$, where
$$A = (A_1,A_2),\,\,B = (B_1,B_2),\,\,C = (C_1,C_2),\,\,D = (D_1,D_2), $$
their Gram determinant is
$$
\det(A,B,C,D) := \det \left(\begin{array}{cccc} {\rm Re}\,A_1 &{\rm Im}\,A_1 &{\rm Re}\,A_2 &{\rm Im}\,A_2\\
{\rm Re}\,B_1 &{\rm Im}\,B_1 &{\rm Re}\,B_2 &{\rm Im}\,B_2\\
{\rm Re}\,C_1 &{\rm Im}\,C_1 &{\rm Re}\,C_2 &{\rm Im}\,C_2\\
{\rm Re}\,D_1 &{\rm Im}\,D_1 &{\rm Re}\,D_2 &{\rm Im}\,D_2\end{array}\right).
$$

Each tetrahedron $NSP_{i}P_{i+1}$ with $i\in\{1,\ldots,2n+1\}$ is isometric to $NSP_{2n + i + 1}P_{2n + i + 2}$, $i\in\{1,\ldots,2n+1\}$, $P_{4n+3} = P_1$ by means of the isometry $C$ defined above. Thus, we consider only the tetrahedra $NSP_{i}P_{i+1}$ with $i\in\{1,\ldots,2n+1\}$.  Split them into two groups: the tetrahedra $NSP_{2k+1}P_{2k+2}$ with $k\in\{0,\ldots,n\}$ and the tetrahedra $NSP_{2k}P_{2k+1}$ with $k\in\{1,\ldots,n\}$.

Substitute $\alpha=\beta+\pi$ and proceed with straightforward calculations:
$$\Delta^{(1)}_{k}(\beta) = \det(S,N,P_{2k+1},P_{2k+2}) = \cos^2\frac{L_1\,\beta}{4} - U^2_{2k-1}(\cos\theta)\,\sin^2\frac{\beta}{2}=$$
$$=T^2_{L_1}(\cos\frac{\beta}{4}) - U^2_{2k-1}(\cos\theta)\,\sin^2\frac{\beta}{2},$$
where $k\in\{0,\ldots,n\}$, $L_1=|2n-4k+1|$, $\theta=\frac{\pi}{2n+1}$, $\beta\in[-2\,\theta,2\,\theta]$;
$$\Delta^{(2)}_{k}(\beta) = \det(S,N,P_{2k},P_{2k+1}) = \cos^2\frac{L_2\,\beta}{4} - U^2_{2k-2}(\cos\theta)\,\sin^2\frac{\beta}{2}=$$
$$=T^2_{L_2}(\cos\frac{\beta}{4}) - U^2_{2k-1}(\cos\theta)\,\sin^2\frac{\beta}{2},$$
where $k\in\{1,\ldots,n\}$, $L_2=|2n-4k+3|$, $\theta$ and $\beta$ the same as above. The first kind Chebyshev polynomial  of degree $k\geq0$ is denoted by $T_k$. Assume that
$$U_{-1}(\cos \theta) = 0,\,\,\,U_{0}(\cos \theta) = 1$$ for the sake of brevity.

All the functions $\Delta^{(j)}_k(\beta)$, $j\in\{1,2\}$ are even on the interval $[-2\theta,2\theta]$. Then one considers them only on the interval $[0,2\theta]$. Note, that the polynomial $T^2_{L_j}(\cos\beta)$ monotonously decreases and the function $\sin^2\frac{\beta}{2}$ monotonously increases with $\beta\in[0,2\theta]$. Moreover, $T^2_{L_j}(\cos 0)=T^2_{L_j}(1)=1$. Then it follows that $\Delta^{(j)}_k(\beta) > 0$ with $\beta\in(-2\theta,2\theta)$. Also, one has $\Delta^{(j)}_k(\pm2\,\theta) = 0$.

Then for all $\beta\in(-2\theta,2\theta)$ (i.e. for all $\alpha$ in the affirmation of the Theorem)
$$ \det(S,N,P_i,P_{i+1}) > 0 $$
where $i\in\{1,\ldots,4n+2\}$, $P_{4n+3}=P_1$.
Thus, claim (e) for the polyhedron $\mathcal{P}_n$ is satisfied.

Use the Schl\"afli formula \cite{Ho} to obtain the volume formula for $\mathbb{T}_n(\alpha)$. One has
$$ {\rm d}{\rm \mathbb{V}ol}\,\mathbb{T}_n(\alpha)\,=\,\frac{l_{\alpha}}{2}\,{\rm d}\alpha\,=\,\frac{(2n+1)\alpha - (2n-1)\pi}{2}\,{\rm d}\alpha. $$

Note, that ${\rm \mathbb{V}ol}\,\mathbb{T}_n(\alpha) \rightarrow 0$ with $\alpha \rightarrow \frac{2n-1}{2n+1}\,\pi$. In this case $d_{\lambda}(P_i,P_{i+1}) \rightarrow 0$, where $i\in\{1,\ldots,4n+2\}$, $P_{4n+3}=P_1$ and the fundamental polyhedron collapses to a point. Thus
$$ {\rm \mathbb{V}ol}\,\mathbb{T}_n(\alpha)\,=\,\frac{1}{2n+1}\,\left(\frac{2n+1}{2}\,\alpha\,-\,\frac{2n-1}{2}\,\pi\right)^2.$$
\end{proof}

\begin{remark}
The domain of the spherical metric existence in Theorem \ref{tor_knot} was indicated before in \cite[Proposition~2.1]{Po}.
\end{remark}

\section{Torus links $\mathbb{L}_n$}

Let $\mathbb{L}_n, n\geq2$ be a torus link ${\rm t}(2n,2)$ with two components. The corresponding link in the rational census is $2n/1$ (Fig.~\ref{Fig3}). The fundamental group of $\mathbb{L}_n$ is presented as
$$\pi_1(\mathbb{S}^3\backslash\mathbb{L}_n) = \langle a,b|(ab)^n\,=\,(ba)^n \rangle.$$

\begin{figure}[ht]
\begin{center}
\includegraphics* [totalheight=4cm]{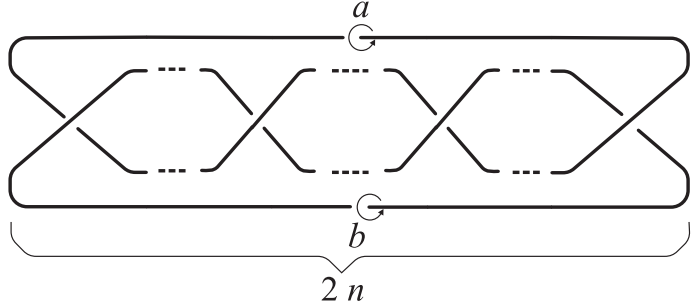}
\end{center}
\caption{Link $2n/1$} \label{Fig3}
\end{figure}

Let $\mathbb{L}_n(\alpha,\,\beta)$ denote a cone-manifold with singular locus the link $\mathbb{L}_n$ and the cone angles $\alpha$, $\beta$ along its components.

For every $\alpha,\beta\in(0,2\pi)$ and $\lambda\in(-1,+1)$,
we denote
$$A=\left(\begin{array}{ccc}1 & 0  \\-2\,i\,e^{i \frac{\alpha}{2}}\lambda\,\sin\frac{\alpha}{2} &e^{i\alpha} \end{array}\right)$$ and
$$B=\left(\begin{array}{ccc}e^{i\beta}& -2\,i\,e^{i \frac{\beta}{2}}\lambda\,\sin\frac{\beta}{2}  \\ 0 & 1 \end{array}\right).$$
By Lemma~\ref{isometry_lemma1}, linear transformations $A$ and $B$ are isometries of $\mathbb{S}^3_{\lambda}$.

\begin{lemma} \label{polynomial_lemma2}
For every integer $n\geq2$ the following equality holds
$$
(AB)^n-(BA)^n\,=\,4\,U_{n-1}(\Lambda)\,\,\lambda\,\,e^{i(\frac{\alpha+\beta}{2}+\pi)\,n}\,\sin\frac{\alpha}{2}\sin\frac{\beta}{2}\,\,M,
$$
where $M$ is a non-zero $2\times2$ matrix and $U_{n-1}(\Lambda)$ is the second kind Chebyshev polynomial of degree $n-1$ in variable $$\Lambda\,=\,(1-\lambda^2)\cos\frac{\alpha-\beta}{2}\,+\,\lambda^2 \cos\frac{\alpha+\beta}{2}\,.$$
\end{lemma}
\begin{proof}.
By analogy with Lemma \ref{polynomial_lemma1}.
\end{proof}

With Lemma \ref{polynomial_lemma2} the main theorem of the section follows:
\begin{theorem} \label{tor_link}
The cone-manifold $\mathbb{L}_n(\alpha,\,\beta)$, $n\geq 2$ is spherical if
$$-2\pi\left(1-\frac{1}{n}\right)\,<\,\alpha\,-\,\beta\,<\,2\pi\left(1-\frac{1}{n}\right)\,,$$
$$2\pi\left(1-\frac{1}{n}\right)\,<\,\alpha\,+\,\beta\,<\,2\pi\left(1+\frac{1}{n}\right)\,.$$
The lengths $l_{\alpha}$, $l_{\beta}$ of its singular geodesics (i.e. lengths of the components for~$\mathbb{L}_n$) are equal to each other and
$$l_{\alpha}\,=\,l_{\beta}\,=\,\frac{\alpha+\beta}{2}\,n\,-\,\pi\,(n-1).$$
The volume of $\mathbb{L}_n(\alpha,\beta)$ is
$${\rm \mathbb{V}ol}\,\mathbb{L}_n(\alpha,\,\beta)\,=\,\frac{1}{2n}\,\left(\frac{\alpha+\beta}{2}\,\,n\,-\,(n-1)\pi\right)^2.$$
\end{theorem}

\begin{proof}.
One continues the proof by analogy with Theorem \ref{tor_knot}.

Suppose that $\mathbb{L}_n(\alpha,\beta)$ is spherical. Then there exists a holonomy mapping~\cite{T2}:
$$h: \pi_1(\mathbb{S}^3\backslash\mathbb{L}_n) \longmapsto {\rm Isom}\,\,\mathbb{S}^3_{\lambda},$$
$$h(a) = A,\,\,h(b) = B.$$

\begin{figure}[ht]
\begin{center}
\includegraphics* [totalheight=8cm]{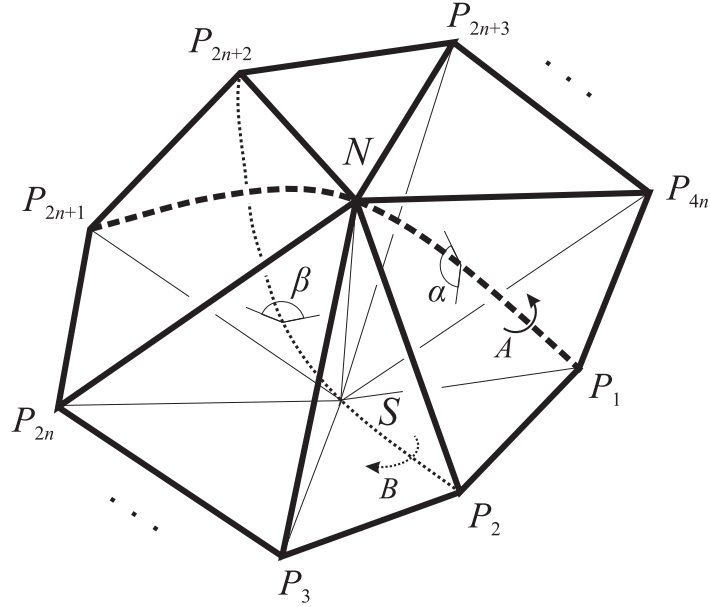}
\end{center}
\caption{The fundamental polyhedron $\mathcal{F}_n$ for $\mathbb{L}_n(\alpha,\beta)$} \label{Fig4}
\end{figure}

Also,
$$ h((ab)^n) - h((ba)^n) = (AB)^n - (BA)^n = 0.$$
By means of Lemma \ref{polynomial_lemma2} the equality above holds either if $\lambda=0$, or if
$$\Lambda\,=\,(1-\lambda^2)\cos\frac{\alpha-\beta}{2}\,+\,\lambda^2 \cos\frac{\alpha+\beta}{2} $$
is a root of the equation $U_{n-1}(\Lambda) = 0.$

In case $\lambda=0$ the image of $h$ is abelian, because of the additional relation
$AB=BA$. With $n\geq2$ this leads to a degenerate geometric structure. Thus, one has to choose the parameter $\lambda$ for the metric ${\rm d}s_{\lambda}^2$ using roots of the Chebyshev polynomial $U_{n-1}(\Lambda)$.

The fundamental polyhedron $\mathcal{F}_n$ for the cone-manifold $\mathbb{L}_n(\alpha,\,\beta)$ is depicted at Fig.~\ref{Fig4}. Suppose its vertices $P_1$ and $P_2$ to be
$$ P_1 = (1,0),\,P_2 = (0,1).$$
The axes of isometries $A$ and $B$ correspond to the edges $P_1P_{2n+1}$ and $P_2P_{2n+2}$. Points $N$ and $S$ are respective middles of the edges $P_1P_{2n+1}$ and $P_2P_{2n+2}$. Those are called North and South poles of the polyhedron.

The polyhedron $\mathcal{F}_n$ is said to be proper if
\begin{list}{}{}
\item (a) respective inner dihedral angles along the edges $P_1P_{2n+1}$ and $P_2P_{2n+2}$  are equal to $\alpha$ and $\beta$;
\item (b) curvilinear faces of the polyhedron are identified by $A$ and $B$:
$$A : NP_1P_2\ldots P_{2n+1} \rightarrow NP_1P_{4n}\ldots P_{2n+2}P_{2n+1}, $$
$$B : SP_2P_1P_{4n}\ldots P_{2n+2} \rightarrow SP_2P_3\ldots P_{2n+2}; $$
\item (c) sum of the inner dihedral angles $\psi_i$ along the edges $P_iP_{i+1}$, $i\in~\{1,\ldots,4n-1\}$ equals $2\pi$;
\item (d) sum of the dihedral angles $\phi_i$ for tetrahedra $NSP_iP_{i+1}$, $i\in\{1,\ldots,4n-1\}$ at their common edge $NS$ equals $2\pi$;
\item (e) all the tetrahedra $NSP_iP_{i+1}$ with $i\in\{1,\ldots,4n\}$, $P_{4n+1}=P_1$ are non-degenerated and coherently oriented.
\end{list}

In order to choose the parameter $\lambda$ for the corresponding metric consider the fundamental polyhedron $\mathcal{F}_n$ with $\alpha=\beta=\pi$. Then all its vertices belong to the same circle and all the dihedral angles $\psi_i$ of the tetrahedra $NSP_iP_{i+1}$ along the edges $P_iP_{i+1}$ are equal to $\psi = \frac{\pi}{2n}\,$ \cite{MR1}. Also the dihedral angles $\phi_i$ of the tetrahedra $NSP_iP_{i+1}$ along their common edge $NS$ are equal to each other:
$$ \phi_i = \phi = \frac{\pi}{2n}.$$

In this case $\lambda = \langle P_1,P_2 \rangle = \cos\phi$ and
$$
\Lambda\,=\,-\cos2\phi\,=\cos\frac{(n-1)\pi}{n}.
$$

All the roots of $U_{n-1}(\Lambda)$ are given by the formula
$$
\Lambda_k\,=\,\cos\frac{k\pi}{n}\,,\,\,k\in\{1,\dots,n-1\},
$$
so one choose the root $\Lambda_k$ with $k=n-1$. Then, by analogy with Theorem \ref{tor_knot}, equalities
$$\sum^{4n}_{i=1}\psi_{i}\,=\,2\pi$$
and
$$\sum^{4n}_{i=1}\phi_{i}\,=\,2\pi$$
are satisfied at the point $\alpha=\beta=\pi$ of the domain
$$ \mathcal{D} = \left\{(\alpha, \beta) : |\alpha-\beta|<2\pi\left(1-\frac{1}{n}\right), |\alpha+\beta-2\pi|<\frac{2\pi}{n}\right\},$$
depicted at Fig.~\ref{Fig5}.

In terms of the parameter $\lambda$, that defines the metric ${\rm d}s_{\lambda}^2$, one has
$$
\lambda^2\,=\,\frac{\cos\frac{\alpha-\beta}{2}+\cos\frac{\pi}{n}}{\cos\frac{\alpha-\beta}{2}-\cos\frac{\alpha+\beta}{2}}\,.
$$

As for all $(\alpha,\beta)\in\mathcal{D}$ the inequality $0<\lambda^2<1$ is satisfied, the metric ${\rm d}s_{\lambda}^2$ is spherical regarding the corresponding domain. By analogy with Theorem \ref{tor_knot} one can show that claims (a)~-- (d) for the polyhedron $\mathcal{F}_n$ are satisfied in the interior of $\mathcal{D}$.

\begin{figure}[t]
\begin{center}
\includegraphics [totalheight=6cm]{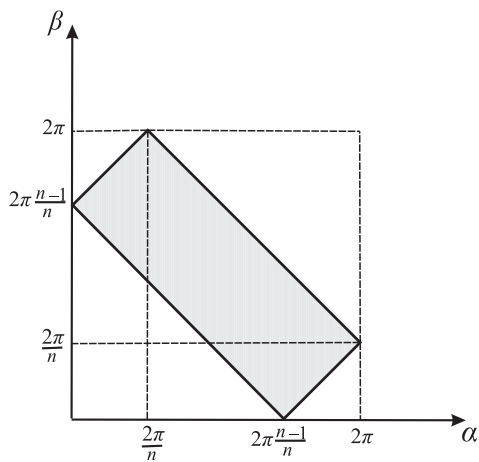}
\end{center}
\caption{The domain $\mathcal{D}$ of sphericity for $\mathbb{L}_n(\alpha,\beta)$} \label{Fig5}
\end{figure}

The lengths $l_{\alpha}$ and $l_{\beta}$ of singular geodesics for the cone-manifold $\mathbb{L}_n(\alpha,\beta)$ meet the relations
$$
\cos \frac{l_{\alpha}}{2} = \langle P_1,N \rangle,
$$
$$
\cos \frac{l_{\beta}}{2} = \langle P_2,S \rangle.
$$

By analogy with the proof of Theorem \ref{tor_knot} one obtains
$$
l_{\alpha}\,=\,l_{\beta}\,=\,\frac{\alpha+\beta}{2}\,n\,-\,\pi(n-1).
$$

Given the coordinates of the vertices for the fundamental polyhedron verify claim (e) for all $(\alpha,\beta)$ in the domain $\mathcal{D}$.

Make use of the Schl\"afli formula \cite{Ho} to obtain the volume of $\mathbb{L}_n(\alpha,\beta)$:
$$ {\rm d}\,{\rm \mathbb{V}ol}\,\mathbb{L}_n(\alpha,\beta)\,=\,\frac{l_{\alpha}}{2}\,{\rm d} \alpha \,+\,\frac{l_{\beta}}{2}\,{\rm d} \beta\,=\,\left(\frac{\alpha+\beta}{2}\,n\,-\,\pi(n-1)\right)\,{\rm d}\left(\frac{\alpha+\beta}{2}\right). $$
Note, that with
$$\alpha=\beta\rightarrow\pi\,\frac{n-1}{n}$$
the fundamental polyhedron $\mathcal{F}_n$ collapses to a point (i.e. the volume tends to~0). The last affirmation of the Theorem follows.
\end{proof}

\begin{remark}
Under condition $\alpha = \beta$ the inequality from the affirmation of Theorem \ref{tor_link} coincides with the inequality from \cite[Proposition 2.2]{Po}. \end{remark}

\begin{remark}
Note, that the lengths of the singular geodesics for $\mathbb{L}_n(\alpha,\beta)$ are equal even if $\alpha \neq \beta$.
\end{remark}

\flushleft{\emph{
Alexander Kolpakov\\
Novosibirsk State University\\
630090, Pirogova str., bld.~2\\
Novosibirsk, Russia\\}
\rm{kolpakov.alexander@gmail.com}
}
\medskip
\flushleft{\emph{
Alexander Mednykh\\
Sobolev Institute of Mathematics, SB RAS\\
630090, Koptyug avenue, bld.~4,\\
Novosibirsk, Russia\\
}
\rm{mednykh@math.nsc.ru}
}
\end{document}